\documentclass[a4paper,10pt]{article}
\usepackage{amsmath,amssymb,amsfonts,amsthm}
\usepackage{fouridx}
\usepackage{a4wide}
\usepackage[usenames,dvipsnames]{color}
\usepackage[verbose,colorlinks=true,linktocpage=true,linkcolor=blue,citecolor=blue]{hyperref}

\theoremstyle{definition}
\newtheorem{definition}{Definition}[section]

\theoremstyle{plain}
\newtheorem{theorem}[definition]{Theorem}
\newtheorem{proposition}[definition]{Proposition}
\newtheorem{lemma}[definition]{Lemma}
\newtheorem{corollary}[definition]{Corollary}

\theoremstyle{remark}
\newtheorem{remark}[definition]{Remark}

\numberwithin{equation}{section}

\newcommand{\Uqqn}{\mathrm{U}_q(\mathfrak{q}_n)}

\newcommand{\Uq}{\mathrm{U}_q}
\newcommand{\Uqi}{\mathrm{U}_{q^{-1}}}
\newcommand{\UA}{\mathrm{U}_{\mathbf{A}_1}}

\allowdisplaybreaks[1]

\begin{document}

\title{Howe Duality for Quantum Queer Superalgebras}
\author{Zhihua Chang ${}^1$ and Yongjie Wang ${}^2$}
\maketitle

\begin{center}
\footnotesize
\begin{itemize}
\item[1] School of Mathematics, South China University of Technology, Guangzhou 510640, China.
\item[2] Mathematics Division, National Center for Theoretical Sciences, Taipei, 10617, Taiwan.
\end{itemize}
\end{center}

\begin{abstract}
We establish a new Howe duality between a pair of quantum queer superalgebras ($\mathrm{U}_{q^{-1}}(\mathfrak{q}_n)$, $\Uq(\mathfrak{q}_m)$). The key ingredient is the construction of a non-commutative analogue $\mathcal{A}_q(\mathfrak{q}_n,\mathfrak{q}_m)$ of the symmetric superalgebra $S(\mathbb{C}^{mn|mn})$ with the use of quantum coordinate queer superalgebra. It turns out that this superalgebra is equipped with a $\Uqi(\mathfrak{q}_n)\otimes\Uq(\mathfrak{q}_m)$-supermodule structure that admits a multiplicity-free decomposition. We also show that the $(\Uqi(\mathfrak{q}_n),\Uq(\mathfrak{q}_m))$-Howe duality implies the Sergeev-Olshanski duality.
\bigskip

\bigskip

\end{abstract}

\section{Introduction}
\label{sec:intr}
The Howe dualities for classical Lie algebras are originated in R. Howe's remarkable work \cite{Howe1}, which leads to a nice proof of the fundamental theorems of classical invariant theory. A systematic
and in-depth study of Howe dualities for Lie superalgebras then followed in \cite{CW01,CW00}. Howe duality have been one of the most inspiring themes in representation theory of Lie algebras, Lie superalgebras and quantum groups. With the use of the Howe dualities, the highest weight vectors were explicitly described in \cite{CW01,CW00}, and some combinatorial character formulas for oscillator representations of orthosymplectic Lie superalgebras were derived in \cite{CZ04}. 

The Howe duality for quantum linear groups $\big(\mathrm{U}_q(\mathfrak{gl}_m), \mathrm{U}_q(\mathfrak{gl}_n)\big)$ was established by R. Zhang \cite{Zhang}. The key point is to construct a suitable non-commutative coordinate algebra for quantum linear groups, which could be viewed as a quantum analogue of a symmetric algebra. The same technique could be applied in constructing the Howe duality of quantum symmetric pairs $\big(\mathrm{U}_q(\mathfrak{gl}_n),\mathrm{U}_q(\mathfrak{so}_{2n})\big)$, $\big(\mathrm{U}_q(\mathfrak{gl}_n),\mathrm{U}_q(\mathfrak{so}_{2n+1})\big)$ and $\big(\mathrm{U}_q(\mathfrak{gl}_n),\mathrm{U}_q(\mathfrak{sp}_{2n})\big)$ in \cite{LZ}, and quantum supergroups $\big(\mathrm{U}_q(\mathfrak{gl}_{m|n}),\mathrm{U}_q(\mathfrak{gl}_{k|l})\big)$ \cite{WuZhang, Zhangyang}. 
As an application, the first and second fundamental theorems of the invariant theory for the quantum general linear (super)groups were obtained in \cite{LZZ,Zhangyang}. 
 
In this paper, we aim to establish a Howe duality for a pair of quantum queer superalgebras. The quantum queer superalgebra $\Uq(\mathfrak{q}_n)$ was constructed in G. Olshanski's letter \cite{Olshanski}, in which the queer analogue of the celebrated Schur-Weyl duality, often referred to as Sergeev-Olshanski duality, was also established. Recently, the highest weight representation theory for quantum queer superalgebra $\Uqqn$ was investigated 
by D. Grantcharov, J. Jung, S-J, Kang and M. Kim \cite{GJKK}. They proved that every $\Uq(\mathfrak{q}_n)$-supermodule in the category of tensor supermodules (sub-supermodules of a finite tensor product of the contravariant supermodule) is completely reducible and all irreducible objects in this category are irreducible highest weight supermodules.
A surprising observation in \cite{GJKK} is that an irreducible $\Uq(\mathfrak{q}_n)$-supermodule over the field $\mathbb{C}(q)$ may be reducible when taking the classical limits. This obstacle will be overcome by enlarging $\mathbb{C}(q)$ to the field $\mathbb{C}((q))$ of formal power series as in \cite{GJKKK}. 

Our approach to obtaining a Howe duality for a pair of quantum queer superalgebras takes advantage of quantum coordinate superalgebras $\mathcal{A}_q(\mathfrak{q}_n)$ as a non-commutative analogue of the symmetric superalgebra $S(\mathbb{C}^{n^2|n^2})$, which is inspired by R. Zhang's paper \cite{Zhang}. We show that the quantum coordinate superalgebra $\mathcal{A}_q(\mathfrak{q}_n)$ has a multiplicity-free decomposition as $\mathrm{U}_{q^{-1}}(\mathfrak{q}_n)\otimes\Uqqn$-supermodule, based on which the Howe duality for a pair of quantum queer superalgebras $(\mathrm{U}_{q^{-1}}(\mathfrak{q}_n),\Uq(\mathfrak{q}_m))$ is obtained. Finally, we recover the Sergeev-Olshanski duality by explicitly defining an action of the Hecke-Clifford superalgebra $\mathrm{HC}_q(m)$ on the zero weight spaces of the supermodules involved in the $(\Uqi(\mathfrak{q}_n),\Uq(\mathfrak{q}_m))$-Howe duality. The invariant theory for quantum queer superalgebras is under consideration and will be treated in our sequel papers.

\section{Quantum queer Lie superalgebra $\Uq(\mathfrak{q}_n)$}

In this section, we will review some basic definitions of the queer Lie superalgebras and their quantum enveloping superalgebras to fix notations. We also briefly summarize a few key facts about the highest weight representation theory of the quantum queer superalgebras that has been systematically investigated in \cite{GJKK}.

For a positive integer $n$, we set $I_{n|n}:=\{-n,\ldots,-1,1,\ldots,n\}$, on which we define the parity of $i\in I_{n|n}$ to be 
$$|i|:=\begin{cases}\bar{0},&\text{if }i>0,\\\bar{1},&\text{if }i<0.\end{cases}$$
Let $V:=\mathbb{C}^{n|n}$ be the superspace with standard basis $v_i$ of parity $|i|$ for $i\in I_{n|n}$. Its endomorphism ring $\mathrm{End}(V)$ is an associative superalgebra with standard basis $E_{ij}$ of parity $|i|+|j|$ for $i,j\in I_{n|n}$. Under the standard supercommutator, $\mathrm{End}(V)$ is also a Lie superalgebra that is denoted by $\mathfrak{gl}_{n|n}$. \textit{The queer Lie superalgebra} $\mathfrak{q}_n$ is the Lie sub-superalgebra of $\mathfrak{gl}_{n|n}$ spanned by
$$e_{ij}:=E_{ij}+E_{-i,-j},\quad f_{ij}:=E_{-i,j}+E_{i,-j}\quad\text{ for }i,j=1,\ldots,n.$$
In the queer Lie superalgebra $\mathfrak{q}_n$, we fix \textit{the standard Cartan sub-superalgebra} $\mathfrak{h}:=\mathfrak{h}_{\bar{0}}\oplus\mathfrak{h}_{\bar{1}}$, where $\mathfrak{h}_{\bar{0}}$ (resp. $\mathfrak{h}_{\bar{1}}$) is spanned by $\epsilon_i^{\vee}:=e_{ii}$ (resp. $f_{ii}$) for $i=1,\ldots, n$. Let $\{\epsilon_1,\ldots,\epsilon_n\}$ be the basis of $\mathfrak{h}_0^*$ dual to $\{\epsilon^{\vee}_1,\ldots, \epsilon^{\vee}_n\}$. Then $P_n:=\oplus_{i=1}^n\mathbb{Z}\epsilon_i$ is \textit{the weight lattice of type $Q$} and $P^{\vee}_n:=\oplus_{i=1}^n\mathbb{Z}\epsilon_i^{\vee}$ is \textit{the dual weight lattice of type $Q$}.

In order to work with the quantization, our base field $\mathbb{C}$ is extended to the field\footnote{The highest weight representation theory of $\Uq(\mathfrak{q}_n)$ over $\mathbb{C}(q)$ was developed in \cite{GJKK}.  The challenge of working over the field $\mathbb{C}(q)$ is that the classical limit of an irreducible highest $\Uq(\mathfrak{q}_n)$-supermodule may no longer be irreducible. As indicated in \cite{GJKKK}, enlarging the base field to $\mathbb{C}((q))$ will overcome this challenge.} $\mathbb{C}((q))$ of formal Laurent series in an indeterminate $q$. We denote $V_q$ to be the $\mathbb{C}((q))$-vector space $V\otimes\mathbb{C}((q))$. The quantum queer superalgebra $\Uq(\mathfrak{q}_n)$ was firstly introduced by G. Olshanski in \cite{Olshanski} with the FRT formulism. The associated $S$-matrix is given by
\begin{align*}
S:=\sum\limits_{i,j\in\rm I }q^{\varphi(i,j)}E_{ii}\otimes E_{jj}+\xi\sum\limits_{i,j\in\rm I, i<j }(-1)^{|i|}(E_{ji}+E_{-j,-i})\otimes E_{ij}\in\mathrm{End}_{\mathbb{C}((q))}(V_q)^{\otimes 2},
\end{align*}
where
$$\varphi(i,j)=(-1)^{|j|}(\delta_{i,j}+\delta_{i,-j}) \text{ and }\xi=q-q^{-1}.$$

\begin{definition}[G. Olshanski \cite{Olshanski}]
The quantum queer superalgebra $\Uq(\mathfrak{q}_n)$ is the unital associative superalgebra over $\mathbb{C}((q))$ generated by elements $L_{ij}$ of parity $|i|+|j|$ for $i,j\in I_{n|n}$ and $i\leqslant j$, with defining relations:
\begin{eqnarray}
&L_{ii}L_{-i,-i}=1=L_{-i,-i}L_{ii},&\\
&L^{12}L^{13}S^{23}=S^{23}L^{13}L^{12},&\label{eq:SLL}
\end{eqnarray}
where $L=\sum\limits_{i\leqslant j} L_{ij}\otimes E_{ij}$ and the relation (\ref{eq:SLL}) holds in $\Uqqn\otimes_{\mathbb{C}((q))}\left(\mathrm{End}_{\mathbb{C}((q))}(V_q)\right)^{\otimes 2}$. 
\end{definition}
The relation (\ref{eq:SLL}) is equivalently rewritten in terms of generators as:
\begin{align*}
&q^{\varphi(j,l)}(-1)^{(|i|+|j|)(|k|+|l|)}L_{ij}L_{kl}
+\delta_{k\leqslant j<l}\xi\theta(i,j,k)L_{il}L_{kj}\\
&+\delta_{i\leqslant-l<j\leqslant-k}\xi\theta(-i,-j,k)L_{i,-l}L_{k,-j}\\
=&q^{\varphi(i,k)}L_{kl}L_{ij}
+\delta_{k<i\leqslant l}\xi\theta(i,j,k) L_{il}L_{kj}\\
&+\delta_{-j\leqslant k<-i\leqslant l}\xi\theta(-i,-j,k) L_{-i,l}L_{-k,j},
\end{align*}
for $i\leqslant j$ and $k\leqslant l$, where $\theta(i,j,k)=(-1)^{|i||j|+|j||k|+|k||i|}$.
\bigskip

Moreover, $\Uq(\mathfrak{q}_n)$ is a Hopf superalgebra with the following comultiplication
$$\Delta(L_{ij})=\sum\limits_{i\leqslant k\leqslant j}(-1)^{(|i|+|k|)(|k|+|j|)}L_{ik}\otimes L_{kj}
=\sum\limits_{i\leqslant k\leqslant j}L_{ik}\otimes L_{kj}\quad\text{ for }i\leqslant j.$$
The counit and antipode on $\Uq(\mathfrak{q}_n)$ are given by $\varepsilon(L)=1$ and $S(L)=L^{-1}$, respectively.

An alternative presentation of the quantum queer superalgebra $\Uq(\mathfrak{q}_n)$ in terms of generators and relations are described in \cite{GJKKK}, where the generators are set to be
\begin{align*}
k_i:&=L_{ii},\quad k_i^{-1}:=L_{-i,-i}, &
e_j:&=-\xi L_{j+1,j+1}L_{-j-1,-j},&
f_j:&=\xi^{-1}L_{j,j+1}L_{-j-1,-j-1},\\
\bar{k}_i:&=-\xi^{-1}L_{-i,i},&
\bar{e}_j:&=-\xi^{-1}L_{j+1,j+1}L_{-j-1,j},&
\bar{f}_j:&=-\xi^{-1}L_{-j,j+1}L_{-j-1,-j-1},
\end{align*}
for $i=1,\ldots,n$ and $j=1,\ldots,n-1$, and the defining relations are given in \cite[Definition~1.1]{GJKKK}.
\bigskip

Next, we briefly review the highest weight representation theory of $\Uq(\mathfrak{q}_n)$ over $\mathbb{C}((q))$ considered in \cite{GJKK, GJKKK}. The superalgebra $\Uq(\mathfrak{q}_n)$ admits the triangular decomposition:
$$\Uq(\mathfrak{q}_n)=\Uq^{-}(\mathfrak{q}_n)\otimes\Uq^0(\mathfrak{q}_n)\otimes\Uq^+(\mathfrak{q}_n),$$
where $\Uq^0(\mathfrak{q}_n)$ is the sub-superalgebra of $\Uq(\mathfrak{q}_n)$ generated by $k_i^{\pm1}$ and $\bar{k}_i$ for $i=1,\ldots,n$, and $\Uq^+(\mathfrak{q}_n)$ (resp. $\Uq^-(\mathfrak{q}_n)$) is the sub-superalgebra of $\Uq(\mathfrak{q}_n)$ generated by $e_i$ and $\bar{e}_i$ (resp. $f_i$ and $\bar{f}_i$) for $i=1,\ldots, n-1$. We also denote by $\Uq^{\geqslant0}(\mathfrak{q}_n)$ (resp. $\Uq^{\leqslant0}(\mathfrak{q}_n)$) the sub-superalgebra of $\Uq(\mathfrak{q}_n)$ generated by $\Uq^0(\mathfrak{q}_n)$ and $\Uq^+(\mathfrak{q}_n)$ (resp. by $\Uq^0(\mathfrak{q}_n)$ and $\Uq^-(\mathfrak{q}_n)$).

A $\Uq(\mathfrak{q}_n)$-supermodule $M$ is called \textit{a weight supermodule} if $M$ admits a weight space decomposition $$M=\bigoplus_{\mu\in P}M_{\mu},$$ 
where $M_{\mu}:=\{m\in M|k_i.m=q^{\mu_i}m\text{ for }i=1,\ldots,n\}$. An element $\mu\in P$ such that $M_{\mu}\neq0$ is called \textit{a weight of $M$}. The set of all weights of $M$ is denoted by $\mathrm{wt}(M)$.

A weight $\Uq(\mathfrak{q}_n)$-supermodule $M$ is called \textit{a highest weight supermodule} if $M$ is generated by a finite-dimensional irreducible $\Uq^{\geqslant0}(\mathfrak{q}_n)$-supermodule $\mathbf{v}$, i.e.,
$$M=\Uq(\mathfrak{q}_n)\otimes_{\Uq^{\geqslant0}(\mathfrak{q}_n)}\mathbf{v}.$$ 
Every nonzero vector in $\mathbf{v}$ is called \textit{a highest weight vector} of $M$. It is known from \cite{GJKK} that 
all highest weight vector of $M$ have the same weight $\lambda\in P$, whence we say that $M$ is a highest weight $\Uq(\mathfrak{q}_n)$-supermodule with highest weight $\lambda$.

Moreover, a finite-dimensional irreducible $\Uq^{\geqslant0}(\mathfrak{q}_n)$-supermodule is determined by a weight $\lambda\in P$ up to the parity reversing functor $\Pi$. The reasoning is the following: Firstly, every finite-dimensional irreducible weight $\Uq^{\geqslant0}(\mathfrak{q}_n)$-supermodule is an irreducible $\Uq^0(\mathfrak{q}_n)$-supermodule with the trivial action by $\Uq^+(\mathfrak{q}_n)$. Secondly, an irreducible $\Uq^0(\mathfrak{q}_n)$-supermodule has weight $\lambda\in P$, and hence an irreducible supermodule of \textit{the quantum Clifford superalgebra} $\mathrm{Cliff}_q(\lambda):=\Uq^0(\mathfrak{q}_n)/I^q(\lambda)$, where $I^q(\lambda)$ is the left ideal of $\Uq^0(\mathfrak{q}_n)$ generated by $h-q^{\lambda(h)}1, h\in P^{\vee}$. Finally,  $\mathrm{Cliff}_q(\lambda)$ has at most two irreducible supermodules $E^q(\lambda)$ and $\Pi E^q(\lambda)$. More precisely, $\mathrm{Cliff}_q(\lambda)$ has a unique irreducible supermodule if $E^q(\lambda)$ is isomorphic to $\Pi E^q(\lambda)$, in which case this irreducible supermodule $E^q(\lambda)$ is of type $Q$\footnote{An irreducible supermodule of type $Q$ (resp. type $M$) if it admits (resp. does not admit) an odd automorphism. A supermodule of type $Q$ (resp. type $M$) is also called self-associate (resp. absolutely irreducible) in literatures (see \cite{BK}).}. While $\mathrm{Cliff}_q(\lambda)$ has exactly two irreducible supermodules $E^q(\lambda)$ and $\Pi E^q(\lambda)$ if they are not isomorphic, in which case both $E^q(\lambda)$ and $\Pi E^q(\lambda)$ are of type $M$. The type of $E^q(\lambda)$ can be distinguished by the number $\ell(\lambda)$ of nonzero components of $\lambda=\lambda_1\epsilon_1+\cdots+\lambda_n\epsilon_n$.
Namely, $E^q(\lambda)$ is of type $M$ if $\ell(\lambda)$ is even and is of type $Q$ if $\ell(\lambda)$ is odd.

\begin{remark}
According to \cite{GJKKK}, we know that $E^q(\lambda)$  over the field $\mathbb{C}((q))$ is of type $M$ if $\ell(\lambda)$ is even. An interesting observation in \cite{GJKK} showed that this may be false if the base field is $\mathbb{C}(q)$.
\end{remark}

\medskip

Conversely, the simple $\mathrm{Cliff}_q(\lambda)$-supermodule $E^q(\lambda)$ is naturally a $\Uq^0(\mathfrak{q}_n)$-supermodule via the canonical quotient map $\Uq^0(\mathfrak{q}_n)\rightarrow\mathrm{Cliff}_q(\lambda)$, which is also regarded as a $\Uq^{\geqslant0}(\mathfrak{q}_n)$-supermodule with the trivial action by $\Uq^+(\mathfrak{q}_n)$. The $\Uq(\mathfrak{q}_n)$-supermodule
$$W^q(\lambda)=\Uq(\mathfrak{q}_n)\otimes_{\Uq^{\geqslant0}(\mathfrak{q}_n)}E^q(\lambda)$$
is call \textit{ the Weyl supermodule} of $\Uq(\mathfrak{q}_n)$ (defined by a highest weight $\lambda$ up to the parity reversing functor $\Pi$). It was demonstrated in \cite{GJKK} that every highest weight $\Uq(\mathfrak{q}_n)$-supermodule of highest weight $\lambda$ is a homomorphic image of $W^q(\lambda)$ and $W^q(\lambda)$ has the unique simple quotient $\mathcal{L}(\lambda)$, which is called \textit{the irreducible highest weight supermodule of $\Uq(\mathfrak{q}_n)$ with highest weight $\lambda$}. We also write $\mathcal{L}(\lambda)$ as $\mathcal{L}^q_n(\lambda)$ if it is necessary to emphasize the superalgebra acting on it.

A naive example of a $\Uq(\mathfrak{q}_n)$-supermodule is the contravariant vector supermodule $V_q$, on which the action of $\Uq(\mathfrak{q}_n)$ is given by
\begin{align*}
k_iv_j&=q^{\delta_{ij}}v_j,&k_iv_{-j}&=q^{\delta_{ij}}v_{-j},&\bar{k}_{i}v_j&=\delta_{j,i}v_{-j},&\bar{k}_{i}v_{-j}&=\delta_{j,i}v_j,&\\
e_iv_j&=\delta_{j,i+1}v_i,&e_iv_{-j}&=\delta_{j,i+1}v_{-i},&f_iv_j&=\delta_{j,i}v_{i+1},&f_{i}v_{-j}&=\delta_{j,i}v_{-i-1},&\\
\bar{e}_{i}v_j&=\delta_{j,i+1}v_{-i},&\bar{e}_{i}v_{-j}&=\delta_{j,i+1}v_{i},&\bar{f}_{i}v_j&=\delta_{j,i}v_{-i-1},&\bar{f}_{i}v_{-j}&=\delta_{j,i}v_{i+1}，&
\end{align*}
for all possible $i,j$.

The $\Uq(\mathfrak{q}_n)$-supermodule $V_q$ is indeed an irreducible highest weight $\Uq(\mathfrak{q}_n)$-supermodule with highest weight $\epsilon_1$, i.e., $V_q\cong \mathcal{L}(\epsilon_1)$. 

Note that $\Uq(\mathfrak{q}_n)$ is a Hopf superalgebra, the tensor product $V_q^{\otimes m}$ for a positive integer $m$ is naturally a $\Uq(\mathfrak{q}_n)$-supermodule with 
$$\mathrm{wt}(V^{\otimes m})\subset P_n^+=\{\lambda=\lambda_1\epsilon_1+\cdots\lambda_n\epsilon_n\in P|\ \lambda_j\in\mathbb{Z}_{\geqslant0}\ \text{ for all } j=1,\ldots n\}.$$
It is known from \cite{GJKKK} that every $V_q^{\otimes m}$ is completely reducible, whose irreducible summands should be irreducible highest weight $\Uq(\mathfrak{q}_n)$-supermodules $\mathcal{L}(\lambda)$ with $\lambda\in \Lambda^+_n\cap P_n^+$, where
$$\Lambda^{+}_n=\{\lambda\in \mathfrak{h}_{\bar{0}}^*|\ \lambda_i-\lambda_{i+1}\in\mathbb{Z}_+\text{ and } \lambda_{i}=\lambda_{i+1}\text
 { implies }\lambda_i=\lambda_{i+1}=0\text{ for all }i=1,\ldots,n-1\}.$$
 Conversely, every $\mathcal{L}(\lambda)$ with $\lambda\in\Lambda_n^+\cap P_n^+$ is finite-dimensional and is an irreducible sub-supermodule of $V^{\otimes m}$ for some positive integer $m$.
 
 The antipode on the Hopf superalgebra $\Uq(\mathfrak{q}_n)$ leads to the natural notion of a dual $\Uq(\mathfrak{q}_n)$-supermodule. Namely, given a $\Uq(\mathfrak{q}_n)$-supermodule $M$, its dual superspace $M^*$ is also a $\Uq(\mathfrak{q}_n)$-supermodule under the action
 $$\langle x.f, v\rangle=(-1)^{|x||f|}\langle f, S(x).v\rangle\quad \text{ for }x\in\Uq(\mathfrak{q}_n),\, f\in M^*,\, v\in M.$$
 For the irreducible highest weight $\Uq(\mathfrak{q}_n)$-supermodule $\mathcal{L}(\lambda)$, its dual supermodule $\mathcal{L}(\lambda)^*$ is an irreducible lowest weight $\Uq(\mathfrak{q}_n)$-supermodule with lowest weight $-\lambda$.
 
We consider the $\mathbb{C}((q))$-semilinear anti-automorphism 
\begin{equation}
\sigma:\Uq(\mathfrak{q}_n)\mapsto\Uq(\mathfrak{q}_n),\quad L_{ij}\mapsto (-1)^{|i||j|+|j|}L_{-j,-i}\quad \text{ for }i,j\in I_{n|n}\text{ with }i\leqslant j,\label{eq:antiauto}
\end{equation}
where the $\mathbb{C}((q))$-semilinearity means that it is $\mathbb{C}$-linear and takes $q$ to $q^{-1}$. Composing with the inverse of the antipode, we obtain a $\mathbb{C}((q))$-semilinear isomorphism  $S^{-1}\circ \sigma: \Uqi(\mathfrak{q}_n)\rightarrow \Uq(\mathfrak{q}_n)$, which induces a functor from the category of $\Uq(\mathfrak{q}_n)$-supermodule to the category of $\Uqi(\mathfrak{q}_n)$-supermodules. The functor maps a $\Uq(\mathfrak{q}_n)$-supermodule $M$ to a $\Uqi(\mathfrak{q}_n)$-supermodule $M^{\sigma}$, on which the action of $\Uqi(\mathfrak{q}_n)$ is given by
$$x\underset{\sigma}{.}v=S^{-1}(\sigma(x)).v\quad \text{ for }x\in \Uqi(\mathfrak{q}_n),\ v\in M.$$
In particular, for the dual irreducible highest weight $\Uq(\mathfrak{q}_n)$-supermodule $(\mathcal{L}^q_n(\lambda))^*$, we have 
\begin{equation}
(\mathcal{L}^q_n(\lambda))^{*\sigma}\cong \mathcal{L}^{q^{-1}}_n(\lambda),\label{eq:Lsigma}
\end{equation}
as $\Uqi(\mathfrak{q}_n)$-supermodules.
\bigskip

We remind the readers that the $\mathrm{Cliff}_q(\lambda)$-supermodule $E^q(\lambda)$ is of type $Q$ if $\ell(\lambda)$ is odd. Consequently, for $\lambda\in\Lambda_n^+\cap P_n^+$ with $\ell(\lambda)$ odd, the $\Uq(\mathfrak{q}_n)$-supermodule $\mathcal{L}(\lambda)$ admits an odd automorphism 
\begin{equation}
\omega_{\lambda}:\mathcal{L}(\lambda)\rightarrow\mathcal{L}(\lambda) 
\label{eq:oddinv}
\end{equation}
such that $\omega_{\lambda}^2=-\mathrm{id}$. For instance, on the contravariant vector supermodule $V_q$, the $\mathbb{C}((q))$-linear map
$$\omega:V_q\rightarrow V_q, \quad v_a\mapsto (-1)^{|a|}v_{-a} \quad \text{ for }a\in I_{n|n}$$
is such an odd automorphism. This leads to the following facts: Let $\lambda\in\Lambda_n^+\cap P_n^+$
\begin{itemize}
\item If $\ell(\lambda)$ is even, then $\mathcal{L}(\lambda)^*\otimes\mathcal{L}(\lambda)$ is an irreducible $\mathcal{L}(\lambda)^*\otimes\mathcal{L}(\lambda)$-supermodule.
\item If $\ell(\lambda)$ is odd, then $\mathcal{L}(\lambda)^*\otimes\mathcal{L}(\lambda)$ is the direct sum of two isomorphic copies of an irreducible $\Uq(\mathfrak{q}_n)\otimes\Uq(\mathfrak{q}_n)$-supermodule.
\end{itemize}
In both cases, we denote the unique irreducible factor of $\mathcal{L}(\lambda)^*\otimes\mathcal{L}(\lambda)$ by $\mathcal{L}(\lambda)^*\circledast\mathcal{L}(\lambda)$.
\bigskip

To conclude this section, we consider the classical limit of $\Uq(\mathfrak{q}_n)$ and its highest weight supermodules as $q\rightarrow1$. Let $\mathbb{C}[[q]]$ be the subring of $\mathbb{C}((q))$ consisting of formal power series in $q$ and
$$\mathbf{A}_1:=\{f/g|\ f,g\in\mathbb{C}[[q]],\ g(1)\neq1\}.$$
Then the $\mathbf{A}_1$-sub-superalgebra $\mathrm{U}_{\mathbf{A}_1}(\mathfrak{q}_n)$ of $\Uq(\mathfrak{q}_n)$ generated by $1, \frac{k_i^{\pm1}-1}{q-1}, \bar{k}_i, e_j, \bar{e}_j, f_j, \bar{f}_j$ for $i=1,\ldots,n$ and $j=1,\ldots,n-1$ is an $\mathbf{A}_1$-form of $\Uq(\mathfrak{q}_n)$. Let $\mathbf{J}_1$ be the ideal of $\mathbf{A}_1$ generated by $q-1$. Then $\mathbf{A}_1/\mathbf{J}_1$ is isomorphic to $\mathbb{C}$. Moreover, the Hopf superalgebra $\mathbf{A}_1/\mathbf{J}_1\otimes_{\mathbf{A}_1}\mathrm{U}_{\mathbf{A}_1}(\mathfrak{q}_n)$ is isomorphic to the universal enveloping superalgebra $\mathrm{U}(\mathfrak{q}_n)$.

The $\mathbf{A}_1$-superalgebra $\mathrm{U}_{\mathbf{A}_1}(\mathfrak{q}_n)$ also admits a triangular decomposition
$$\UA(\mathfrak{q}_n)=\UA^-(\mathfrak{q}_n)\otimes\UA^0(\mathfrak{q}_n)\otimes\UA^+(\mathfrak{q}_n),$$
where $\UA^0(\mathfrak{q}_n)$ is the $\mathbf{A}_1$-sub-superalgebra of $\Uq(\mathfrak{q}_n)$ generated by $\frac{k_i^{\pm1}-1}{q-1},\bar{k}_i$ for $i=1,\ldots,n$, and $\UA^+(\mathfrak{q}_n)$ (resp. $\UA^-(\mathfrak{q}_n)$) is the $\mathbf{A}_1$-sub-superalgebra of $\Uq(\mathfrak{q}_n)$ generated by $e_i$ and $\bar{e}_i$ (resp. $f_i$ and $\bar{f}_i$) for $i=1,\ldots, n-1$. 

Given $\lambda\in\Lambda_+\cap P_+$, the quantum Clifford superalgebra $\mathrm{Cliff}_q(\lambda)=\Uq^0(\mathfrak{q}_n)/I^q(\lambda)$ is generated by the canonical images $\mathbf{1}+I^q(\lambda)$ and $\bar{k}_i+I^q(\lambda)$ for $i=1,\ldots,n$. An $\mathbf{A}_1$-form of $\mathrm{Cliff}_q(\lambda)$ is the $\mathbf{A}_1$-sub-superalgebra of $\mathrm{Cliff}_q(\lambda)$ generated by $\mathbf{1}+I^q(\lambda)$ and $\bar{k}_i+I^q(\lambda)$ for $i=1,\ldots,n$, which we denote by $\mathrm{Cliff}_{\mathbf{A}_1}(\lambda)$. Let $E^q(\lambda)$ be an irreducible $\mathrm{Cliff}_q(\lambda)$-supermodule and $v\in E^q(\lambda)$ be a nonzero even element of $E^q(\lambda)$. We set $E_{\mathbf{A}_1}(\lambda)$ to be the $\mathrm{Cliff}_{\mathbf{A}_1}(\lambda)$-sub-supermodule of $E^q(\lambda)$ generated by $v$. Then $E_{\mathbf{A}_1}(\lambda)$ is an $\mathbf{A}_1$-form of $E^q(\lambda)$ and is invariant under $\UA^0(\mathfrak{q}_n)$. Let $\mathcal{L}(\lambda)$ be the irreducible highest weight $\Uq(\mathfrak{q}_n)$-supermodule generated by $E^q(\lambda)$. Then the $\UA(\mathfrak{q}_n)$-sub-supermodule generated by $E_{\mathbf{A}_1}(\lambda)$ is an $\mathbf{A}_1$-form of $\mathcal{L}(\lambda)$, which is denoted by $\mathcal{L}^{\mathbf{A}_1}(\lambda)$. Moreover, 
$$\mathbf{A}_1/\mathbf{J}_1\otimes_{\mathbf{A}_1}\mathcal{L}^{\mathbf{A}_1}(\lambda)\cong \mathbb{L}(\lambda)$$
as $\mathrm{U}(\mathfrak{q}_n)$-supermodules, where $\mathbb{L}(\lambda)$ is an irreducible highest weight $\mathrm{U}(\mathfrak{q}_n)$-supermodule of highest weight $\lambda$ (see \cite[Proposition~1.9]{GJKKK}).

For the dual supermodule $\mathcal{L}(\lambda)^*$, we define
$$\mathcal{L}^{\mathbf{A}_1}(\lambda)^*=\{f\in\mathcal{L}(\lambda)^*|f\left(\mathcal{L}^{\mathbf{A}_1}(\lambda)\right)\subseteq\mathbf{A}_1\}.$$
Then $\mathcal{L}^{\mathbf{A}_1}(\lambda)^*$ is a $\mathrm{U}_{\mathbf{A}_1}(\mathfrak{q}_n)$-sub-supermodule of $\mathcal{L}(\lambda)^*$ and an $\mathbf{A}_1$-form of $\mathcal{L}(\lambda)$. We also have 
$$\mathbf{A}_1/\mathbf{J}_1\otimes_{\mathbf{A}_1}\mathcal{L}^{\mathbf{A}_1}(\lambda)^*\cong\mathbb{L}(\lambda)^*$$
as $\mathrm{U}(\mathfrak{q}_n)$-supermodule, where $\mathbb{L}(\lambda)^*$ is the dual $\mathrm{U}(\mathfrak{q}_n)$-supermodule of $\mathbb{L}(\lambda)$.

\section{Quantum coordinate superalgebra $\mathcal{A}_q(\mathfrak{q}_n)$}
\label{sec:QCA}

This section is devoted to construct a non-commutative analogue $\mathcal{A}_q(\mathfrak{q}_n)$ of the symmetric superalgebra $S(\mathbb{C}^{n^2|n^2})$. It plays the role of the coordinate superalgebra of the quantum supergroup $\Uq(\mathfrak{q}_n)$. Moreover, $\mathcal{A}_q(\mathfrak{q}_n)$ will be equipped with an action of $\Uq(\mathfrak{q}_n)\otimes\Uq(\mathfrak{q}_n)$ via the left and right translation of $\Uq(\mathfrak{q}_n)$. We will also establish a multiplicity-free decomposition of $\mathcal{A}_q(\mathfrak{q}_n)$ as a $\Uq(\mathfrak{q}_n)\otimes\Uq(\mathfrak{q}_n)$-supermodule that serves as the Peter-Weyl theorem for $\Uq(\mathfrak{q}_n)$.

Let $\Uqqn^{\circ}$ denote the finite dual of the Hopf superalgebra $\Uqqn$, i.e., 
\begin{equation*}
\Uq(\mathfrak{q}_n)^{\circ}:=\left\{f\in\Uq(\mathfrak{q}_n)^*\middle|\ \ker f\text{ contains a cofinite } \mathbb{Z}_2\text{-graded ideal of }\Uqqn\right\},
\end{equation*}
which also has the structure of a Hopf superalgebra. The multiplication, comultiplication, counit and antipode of $U_q(\mathfrak{q}_n)^{\circ}$ will be denoted by $m^{\circ}, \Delta^{\circ}, \epsilon^{\circ}$ and $S^{\circ}$, respectively. 

We define two $\Uq(\mathfrak{q}_n)$-supermodule structures on $\Uq(\mathfrak{q}_n)^{\circ}$:
\begin{align*}
\Phi&:\Uqqn\otimes\Uqqn^{\circ}\rightarrow\Uqqn^{\circ},\quad  x\otimes f\mapsto \Phi_x(f),\\
\Psi&:\Uqqn\otimes\Uqqn^{\circ}\rightarrow\Uqqn^{\circ}, \quad x\otimes f\mapsto \Psi_x(f),
\end{align*}
where $\Phi_x(f), \Psi_x(f)\in\Uq(\mathfrak{q}_n)^{\circ}$ are given by
\begin{equation*}
\langle\Phi_x(f),y\rangle=(-1)^{(|f|+|y|)|x|}\langle f,yx\rangle\text{ and }
\langle\Psi_x(f),y\rangle=(-1)^{|f||x|}\langle f,S(x)y\rangle, 
\end{equation*}
for $y\in\Uq(\mathfrak{q}_n)$. Moreover, we verify that the two actions of $\Uq(\mathfrak{q}_n)$ on $\Uq(\mathfrak{q}_n)^{\circ}$ given by $\Phi$ and $\Psi$ are compatible with the superalgebra structures on $\Uq(\mathfrak{q}_n)^{\circ}$ in the sense that $\Uq(\mathfrak{q}_n)^{\circ}$ is a $\Uq(\mathfrak{q}_n)$-supermodule superalgebra under $\Phi$ and a $\Uq(\mathfrak{q}_n)$-supermodule superalgebra under $\Psi$ with respect to the opposite comultiplication, i.e.,
\begin{eqnarray*}
&\Phi_x(fg)=\sum(-1)^{|f||x_{(2)}|}\Phi_{x_{(1)}}(f)\Phi_{x_{(2)}}(g),\text{ and }&\\
&\Psi_x(fg)=\sum(-1)^{|f||x_{(1)}|+|x_{(1)}||x_{(2)}|}\Psi_{x_{(2)}}(f)\Psi_{x_{(1)}}(g).&
\end{eqnarray*}
Furthermore, the two actions $\Phi$ and $\Psi$ are super-commutative, i.e.,
$$\Phi_x\Psi_y(f)=(-1)^{|x||y|}\Psi_y\Phi_x(f),$$
for $x,y\in\Uq(\mathfrak{q}_n)$ and $f\in\Uq(\mathfrak{q}_n)^{\circ}$. It leads to a $\Uq(\mathfrak{q}_n)\otimes\Uq(\mathfrak{q}_n)$-supermodule structure on $\Uq(\mathfrak{q}_n)^{\circ}$ under the joint action $\Psi\otimes\Phi$, i.e.,
$$(x\otimes y).f=\Psi_x\Phi_y(f)\quad\text{ for } x,y\in\Uq(\mathfrak{q}_n) \text{ and }f\in\Uq(\mathfrak{q}_n)^{\circ}.$$
\bigskip

In order to explore the $\Uq(\mathfrak{q}_n)\otimes\Uq(\mathfrak{q}_n)$-supermodule structure on $\Uq(\mathfrak{q}_n)^{\circ}$, we introduce a $\mathbb{C}((q))$-linear map for each $\lambda\in\Lambda_n^+\cap P_n^+$:
$$\tau^{\lambda}:\mathcal{L}(\lambda)^*\otimes \mathcal{L}(\lambda)\rightarrow \Uq(\mathfrak{q}_n)^{\circ}, \quad \tilde{u}\otimes v\mapsto \tau^{\lambda}_{\tilde{u},v},$$
where $\tau^{\lambda}_{\tilde{u},v}$ is defined by $\tau^{\lambda}_{\tilde{u},v}(x):=(-1)^{|x||v|}\langle \tilde{u},xv\rangle$. The linear functional $\tau^{\lambda}_{\tilde{u},v}$ is contained in $\Uq(\mathfrak{q}_n)^{\circ}$ since $\mathcal{L}(\lambda)$ is finite-dimensional. Moreover,
\begin{equation}
\Phi_x(\tau^{\lambda}_{\tilde{u},v})=(-1)^{|x||\tilde{u}|}\tau^{\lambda}_{\tilde{u},x.v}\ \text{ and }\ \Psi_x(\tau^{\lambda}_{\tilde{u},v})=\tau^{\lambda}_{x.\tilde{u},v},\label{eq:Aqact}
\end{equation}
for $x\in\Uq(\mathfrak{q}_n)$, $\tilde{u}\in\mathcal{L}(\lambda)^*$ and $v\in\mathcal{L}(\lambda)$. Hence, $\tau^{\lambda}$ is a $\Uq(\mathfrak{q}_n)\otimes\Uq(\mathfrak{q}_n)$-supermodule homomorphism.

\begin{lemma}
\label{lem:imagetau}
If $\ell(\lambda)$ is odd, then
$$\tau^{\lambda}_{\tilde{\omega}_{\lambda}(\tilde{u}),\omega_{\lambda}(v)}=-(-1)^{|\tilde{u}|}\tau^{\lambda}_{\tilde{u},v} \quad\text{ for }\ \tilde{u}\in\mathcal{L}(\lambda)^*, v\in\mathcal{L}(\lambda),$$
where $\omega_{\lambda}$ is the odd automorphism of $\mathcal{L}(\lambda)$ as (\ref{eq:oddinv}) and $\tilde{\omega}_{\lambda}$ is the odd automorphism of $\mathcal{L}(\lambda)^*$ induced by $\omega_{\lambda}$. Consequently, the image of $\tau^{\lambda}$ is isomorphic to $\mathcal{L}(\lambda)^*\circledast\mathcal{L}(\lambda)$.
\end{lemma}

\begin{proof}
The odd automorphism $\tilde{\omega}_{\lambda}:\mathcal{L}(\lambda)^*\rightarrow\mathcal{L}(\lambda)^*$ is given by
$$\langle\tilde{\omega}_{\lambda}(\tilde{u}),v\rangle=(-1)^{|\tilde{u}|}\langle\tilde{u},\omega_{\lambda}(v)\rangle \text{ for }\tilde{u}\in\mathcal{L}(\lambda)^*, v\in\mathcal{L}(\lambda).$$

Hence, for $x\in\Uq(\mathfrak{q}_n), \tilde{u}\in\mathcal{L}(\lambda)^*, v\in\mathcal{L}(\lambda)$, we verify that
\begin{align*}
\tau^{\lambda}_{\tilde{\omega}_{\lambda}(\tilde{u}),\omega_{\lambda}(v)}(x)=&(-1)^{|x|(|v|+1)}\langle \tilde{\omega}_{\lambda}(\tilde{u}),x.\omega_{\lambda}(v)\rangle=(-1)^{|x||v|}\langle \tilde{\omega}_{\lambda}(\tilde{u}),\omega_{\lambda}(x.v)\rangle\\
=&-(-1)^{|x||v|+|\tilde{u}|}\langle \tilde{u},x.v\rangle=-(-1)^{|\tilde{u}|}\tau^{\lambda}_{\tilde{u},v}(x),
\end{align*}
i.e., $\tau^{\lambda}_{\tilde{\omega}_{\lambda}(\tilde{u}),\omega_{\lambda}(v)}=-(-1)^{|\tilde{u}|}\tau^{\lambda}_{\tilde{u},v}$.
\medskip

Now, $\tau^{\lambda}:\mathcal{L}(\lambda)^*\otimes\mathcal{L}(\lambda)\rightarrow\Uq(\mathfrak{q}_n)^{\circ}$ is a nonzero homomorphism of $\Uq(\mathfrak{q}_n)\otimes\Uq(\mathfrak{q}_n)$-supermodules. Hence, the image of $\tau^{\lambda}$ is a nonzero quotient of $\mathcal{L}(\lambda)^*\otimes\mathcal{L}(\lambda)$. The equality $\tau^{\lambda}_{\tilde{\omega}_{\lambda}(\tilde{u}),\omega_{\lambda}(v)}=-(-1)^{|\tilde{u}|}\tau^{\lambda}_{\tilde{u},v}$ ensures that the image of $\tau^{\lambda}$ has a dimension strictly less than the dimension of $\mathcal{L}_n(\lambda)^*\otimes\mathcal{L}_n(\lambda)$. Consequently, as a nonzero proper quotient of $\mathcal{L}(\lambda)^*\otimes\mathcal{L}(\lambda)$, the image of $\tau^{\lambda}$ is isomorphic to $\mathcal{L}(\lambda)^*\circledast\mathcal{L}(\lambda)$, the irreducible $\Uq(\mathfrak{q}_n)\otimes\Uq(\mathfrak{q}_n)$-sub-supermodule of $\mathcal{L}(\lambda)^*\otimes\mathcal{L}(\lambda)$.
\end{proof}

Now, we consider the special case where $\mathcal{L}(\epsilon_1)=V_q$ is the contravariant vector supermodule of $\Uq(\mathfrak{q}_n)$. Fix the standard basis $\{v_a, a\in I_{n|n}\}$ of  $V_q$ and the dual basis $\{v_a^*, a\in I_{n|n}\}$ of $V_q^*$, we call 
\begin{equation*}
t_{ab}:=(-1)^{(|a|+|b|)|b|}\tau^{\epsilon_1}_{v_a^*,v_b}\in\Uq(\mathfrak{q}_n)^{\circ}\quad \text{ for }a,b\in I_{n|n}
\end{equation*}
\textit{the matrix elements furnished by the $\Uq(\mathfrak{q}_n)$-supermodule $V_q$.} They satisfy
$$x.v_b=\sum_{a\in I_{n|n}}\langle t_{ab}, x\rangle v_a\quad \text{ for }x\in\Uq(\mathfrak{q}_n).$$

\begin{definition}
The sub-superalgebra of $\Uqqn^{\circ}$ generated by $t_{ab}$ for $a,b\in I_{n|n}$ is called \textit{the quantum coordinate superalgebra of $\Uqqn$}, denoted by $\mathcal{A}_q(\mathfrak{q}_n)$.
\end{definition}

\begin{proposition}
The quantum coordinate superalgebra $\mathcal{A}_q(\mathfrak{q}_n)$ is a sub-bi-superalgebra of $\Uqqn^{\circ}$, in which the generators $t_{ab}$ for $a,b\in I_{n|n}$ satisfy the relations:
\begin{align}
t_{ab}&=t_{-a,-b}\quad \text{for }a,b\in I_{n|n},\label{eq:QCA1}\\
S^{12}T^{13}T^{23}&=T^{23}T^{13}S^{12},\label{eq:QCA2}
\end{align}
where $T=\sum_{a,b\in I_{n|n}}E_{ab}\otimes t_{ab}\in\mathrm{End}(V_q)\otimes_{\mathbb{C}((q))}\mathcal{A}_q(\mathfrak{q}_n)$, and the comultiplication satisfies 
\begin{equation*}
\Delta^{\circ}(t_{ab})=\sum\limits_{c\in I_{n|n}}t_{ac}\otimes t_{cb} \quad\text{ for }a,b\in I_{n|n}.
\end{equation*}
\end{proposition}

\begin{proof}
The proof is straightforward. We omit the details here.
\end{proof}

\begin{remark}
\label{rmk:Aqn}
It follows from (\ref{eq:QCA1}) and (\ref{eq:QCA2}) that $\mathcal{A}_q(\mathfrak{q}_n)$ is spanned by
$$\big\{t_{a_1,b_1}t_{a_{2},b_{2}}\cdots t_{a_l,b_l}|\ a_i\in I_{n|n}, 1\leqslant b_1\leqslant\cdots\leqslant b_l\leqslant n \text{ and } l\geqslant0\big\}.$$
\end{remark}

Moreover, as a result of (\ref{eq:Aqact}), $\mathcal{A}_q(\mathfrak{q}_n)$ is invariant under both the actions of $\Phi$ and $\Psi$, and thus a $\Uq(\mathfrak{q}_n)\otimes\Uq(\mathfrak{q}_n)$-supermodule. Combining the superalgebra structure on $\mathcal{A}_q(\mathfrak{q}_n)$, we obtain that $\mathcal{A}_q(\mathfrak{q}_n)$ is a $\Uq(\mathfrak{q}_n)$-supermodule superalgebra under $\Phi$ and a $\Uq(\mathfrak{q}_n)$-supermodule superalgebra under $\Psi$ with respect to the opposite comultiplication of $\Uq(\mathfrak{q}_n)$. 

\begin{theorem}
\label{thm:Aq_dec}
The $\Uq(\mathfrak{q}_n)\otimes\Uq(\mathfrak{q}_n)$-supermodule $\mathcal{A}_q(\mathfrak{q}_n)$ under $\Psi\otimes\Phi$ admits a multiplicity-free decomposition:
\begin{equation*}
\mathcal{A}_q(\mathfrak{q}_n)\cong
\bigoplus\limits_{\lambda\in\Lambda^+_n\cap P_n^+}
\mathcal{L}(\lambda)^*\circledast \mathcal{L}(\lambda).
\label{eq:PeterWeyl}
\end{equation*}
\end{theorem}

\begin{proof}[Proof of Theorem~\ref{thm:Aq_dec}]
For $\lambda\in\Lambda^+_n\cap P^+_n$, it is known from~\cite{GJKK} that the irreducible highest weight $\Uq(\mathfrak{q}_n)$-supermodule $\mathcal{L}(\lambda)$ is a sub-supermodule of the tensor product of finitely many copies of $V_q$, which implies that $\tau^{\lambda}_{\tilde{u},v}$ for $\tilde{u}\in\mathcal{L}(\lambda)^*$ and $v\in\mathcal{L}(\lambda)$ is generated by $t_{ab}$ for $a,b\in I_{n|n}$. Hence, the image of $\tau^{\lambda}$ is contained in $\mathcal{A}_q(\mathfrak{q}_n)$. 

On the other hand, a tensor product of finitely many copies of $V_q$ is completely reducible. Each of its irreducible summands is of the form $\mathcal{L}(\lambda)$ for $\lambda\in\Lambda^+_n\cap P^+_n$. It follows that $\mathcal{A}_q(\mathfrak{q}_n)$ is spanned by the image of $\tau^{\lambda}$ for $\lambda\in\Lambda^+_n\cap P^+_n$.

Moreover, if $\ell(\lambda)$ is even, then $\mathcal{L}(\lambda)^*\otimes\mathcal{L}(\lambda)$ is irreducible as a $\Uq(\mathfrak{q}_n)\otimes\Uq(\mathfrak{q}_n)$-supermodule. It yields that $\tau^{\lambda}$ is injective, whose image is isomorphic to $\mathcal{L}(\lambda)^*\otimes\mathcal{L}(\lambda)$. 
If $\ell(\lambda)$ is odd, then Lemma~\ref{lem:imagetau} implies that the image of 
$\tau^{\lambda}$ is isomorphic to $\mathcal{L}(\lambda)^*\circledast\mathcal{L}(\lambda)$, the irreducible factors of $\mathcal{L}(\lambda)^*\otimes\mathcal{L}(\lambda)$. 

Finally, the irreducible $\Uq(\mathfrak{q}_n)\otimes\Uq(\mathfrak{q}_n)$-supermodules $\mathcal{L}(\lambda)^*\circledast\mathcal{L}(\lambda)$ for $\lambda\in\Lambda^+_n\cap P^+_n$ are pairwisely nonisomorphic. Hence, $\mathcal{A}_q(\mathfrak{q}_n)$ is the direct sum of the images of $\tau^{\lambda}$ for $\lambda\in\Lambda^+_n\cap P^+_n$ and we obtain the desired decomposition.
\end{proof}

\begin{remark}
Theorem~\ref{thm:Aq_dec} can be viewed as Peter-Weyl theorem for the quantum queer superalgebra $\Uqqn$. 
\end{remark}
\bigskip

Recall from (\ref{eq:antiauto}) that there is a $\mathbb{C}((q))$-semilinear anti-automorphism $\sigma$ on $\Uq(\mathfrak{q}_n)$. It yields a $\Uqi(\mathfrak{q}_n)$-supermodule structure $\tilde{\Psi}$ on $\Uq(\mathfrak{q}_n)^{\circ}$:
$$\langle\tilde{\Psi}_x.f, y\rangle=(-1)^{|x||f|}\langle f, \sigma(x)y\rangle\quad \text{ for }x,\,y\in\Uq(\mathfrak{q}_n) \text{ and } f\in\Uq(\mathfrak{q}_n)^{\circ}.$$
In particular, we verify that
$$\tilde{\Psi}_x.\tau^{\lambda}_{\tilde{u},v}=\tau^{\lambda}_{S\circ \sigma(x).\tilde{u},v}\quad \text{for }x\in\Uq(\mathfrak{q}_n),\, \tilde{u}\in \mathcal{L}(\lambda)^* \text{ and } v\in\mathcal{L}(\lambda),$$
which yields that $\mathcal{A}_q(\mathfrak{q}_n)$ is a sub-supermodule under the action $\tilde{\Psi}$. We verify that $\Uq(\mathfrak{q}_n)^{\circ}$ is also a $\Uqi(\mathfrak{q}_n)$-supermodule superalgebra under the action $\tilde{\Psi}$ and so is $\mathcal{A}_q(\mathfrak{q}_n)$. Then $\mathcal{A}_q(\mathfrak{q}_n)$ is a $\Uqi(\mathfrak{q}_n)\otimes\Uq(\mathfrak{q})$-supermodule since $\tilde{\Psi}$ is also supercommutative with $\Phi$. Combining with (\ref{eq:Lsigma}), we have

\begin{corollary}
The $\Uqi(\mathfrak{q}_n)\otimes\Uq(\mathfrak{q})$-supermodule $\mathcal{A}_q(\mathfrak{q}_n)$ under $\tilde{\Psi}\otimes\Phi$ admits a multiplicity-free decomposition:
\begin{equation}
\mathcal{A}_q(\mathfrak{q}_n)\cong
\bigoplus\limits_{\lambda\in\Lambda^+_n\cap P_n^+}
\mathcal{L}^{q^{-1}}_n(\lambda)\circledast \mathcal{L}^q_n(\lambda).
\label{eq:MfreeDec}
\end{equation}\qed
\end{corollary}
\begin{remark}
\label{rmk:AqA1}
Recall that $\UA(\mathfrak{q}_n)$ is a $\mathbf{A}_1$-form of $\Uq(\mathfrak{q}_n)$, we define
$$\UA(\mathfrak{q}_n)^{\circ}=\{f\in\Uq(\mathfrak{q}_n)^{\circ}|\ f(\UA(\mathfrak{q}_n))\subseteq\mathbf{A}_1\},$$
which turns out to be an $\mathbf{A}_1$-form of $\Uq(\mathfrak{q}_n)^{\circ}$. It is easily observed that $t_{ab}\in\UA(\mathfrak{q}_n)^{\circ}$, and hence, the $\mathbf{A}_1$-sub-superalgebra generated by $t_{ab}$ for $a,b\in  I_{n|n}$ is an $\mathbf{A}_1$-form of $\mathcal{A}_q(\mathfrak{q}_n)$ that is denoted by $\mathcal{A}_{\mathbf{A}_1}(\mathfrak{q}_n)$. It is invariant under the action $\tilde{\Psi}$, $\Psi$ and $\Phi$ of $\UA(\mathfrak{q}_n)$. Taking the classical limit, we obtain that $\mathbf{A}_1/\mathbf{J}_1\otimes_{\mathbf{A}_1}\mathcal{A}_{\mathbf{A}_1}(\mathfrak{q}_n)$ is isomorphic to the symmetric superalgebra $S(\mathbb{C}^{n^2|n^2})$. 
\end{remark}

\section{Howe duality for quantum queer superalgebras}

For two positive integers $m, n$, we set $s=\max(m,n)$. It is obvious that the sub-superalgebra of $\Uq(\mathfrak{q}_s)$ generated by $L_{i,j}$ for $i\leqslant j$ with $i,j\in I_{m|m}$ is a Hopf supersuperalgebra of $\Uq(\mathfrak{q}_s)$ isomorphic to $\Uq(\mathfrak{q}_m)$. Similarly, $\Uqi(\mathfrak{q}_n)$ is identified with the corresponding Hopf sub-superalgebra of $\Uqi(\mathfrak{q}_s)$.  We have shown in Section~\ref{sec:QCA} that $\mathcal{A}_q(\mathfrak{q}_s)$ is a $\Uqi(\mathfrak{q}_s)\otimes\Uq(\mathfrak{q}_s)$-supermodule superalgebra under the joint action $\tilde{\Psi}\otimes\Phi$. Hence, it is naturally a $\Uqi(\mathfrak{q}_n)\otimes\Uq(\mathfrak{q}_m)$-supermodule superalgebra via the restriction of $\tilde{\Psi}$ to $\Uqi(\mathfrak{q}_n)$ and the restriction of $\Phi$ to $\Uq(\mathfrak{q}_m)$. However, $\mathcal{A}_q(\mathfrak{q}_s)$ fails to admit a multiplicity-free decomposition with respect to the action of $\Uqi(\mathfrak{q}_n)\otimes\Uq(\mathfrak{q}_m)$. Instead, we consider the following sub-superspace:
$$\mathcal{A}_q(\mathfrak{q}_n,\mathfrak{q}_m)
:=\{ f\in\mathcal{A}_q(\mathfrak{q}_s)|\ \tilde{\Psi}_{k_i}.f=f \text{ for }n<i\leqslant s\text{ and }\Phi_{k_j}.f=f\text{ for }m<j\leqslant s\}.$$

Note that $\Delta(k_i)=k_i\otimes k_i$ for $i=1,\ldots, s$, the sub-superspace $\mathcal{A}_q(\mathfrak{q}_n,\mathfrak{q}_m)$ of $\mathcal{A}_q(\mathfrak{q}_s)$ is a sub-superalgebra. One easily observes that $t_{ab}$ with $a\in I_{n|n}$ and $b\in I_{m|m}$ are contained in $\mathcal{A}_q(\mathfrak{q}_n,\mathfrak{q}_m)$ and $\mathcal{A}_q(\mathfrak{q}_n,\mathfrak{q}_m)$ is indeed the sub-superalgebra of $\mathcal{A}_q(\mathfrak{q}_s)$ generated by $t_{ab}$ for $a\in I_{n|n}$ and $b\in I_{m|m}$. 
Moreover, $k_i, n<i\leqslant s$ centralize $\Uqi(\mathfrak{q}_n)$ in $\Uqi(\mathfrak{q}_s)$, 
the sub-superspace $\mathcal{A}_q(\mathfrak{q}_n,\mathfrak{q}_m)$ is a $\Uqi(\mathfrak{q}_n)$-supermodule under $\tilde{\Psi}$. It is indeed a $\Uqi(\mathfrak{q}_n)$-supermodule superalgebra under $\tilde{\Psi}$ with respect to the opposite comultiplication on $\Uqi(\mathfrak{q}_n)$ since so is $\mathcal{A}_q(\mathfrak{q}_s)$. Similarly, $\mathcal{A}_q(\mathfrak{q}_n,\mathfrak{q}_m)$ is also a $\Uq(\mathfrak{q}_m)$-supermodule superalgebra under $\Phi$. This gives rise to a $\Uqi(\mathfrak{q}_n)\otimes\Uq(\mathfrak{q}_m)$-supermodule $\mathcal{A}_q(\mathfrak{q}_n,\mathfrak{q}_m)$ under the joint action $\tilde{\Psi}\otimes\Phi$ since $\tilde{\Psi}$ and $\Phi$ are super-commutative.

\begin{remark}
\label{rmk:Aqmn}
According to Remark~\ref{rmk:Aqn}, $\mathcal{A}_q(\mathfrak{q}_n,\mathfrak{q}_m)$ is spanned by
$$\{t_{a_1,b_1}\cdots t_{a_l,b_l}|1\leqslant b_1\leqslant\cdots\leqslant b_l\leqslant m,\ a_i\in I_{n|n}\text{ for }i=1,\ldots,l \text{ and } l\geqslant 0\}.$$
\end{remark}

\begin{theorem}[How duality for quantum queer superalgebras]
\label{thm:howe} 
Let $m,n$ be two positive integers and $r:=\min(m,n)$. Then $\mathcal{A}_q(\mathfrak{q}_n,\mathfrak{q}_m)$ admits the following multiplicity-free decomposition as a $\Uqi(\mathfrak{q}_n)\otimes\Uq(\mathfrak{q}_m)$-supermodule:
\begin{equation*}
\mathcal{A}_q(\mathfrak{q}_n,\mathfrak{q}_m)=
\bigoplus\limits_{\lambda\in\Lambda^+_r\cap P_r^+}
\mathcal{L}^{q^{-1}}_n(\lambda)\circledast \mathcal{L}^q_m(\lambda).
\label{eq:howe}
\end{equation*}
\end{theorem}

In order to prove this theorem, we need the following lemma.
\begin{lemma}
\label{lem:fixpt}
Let $\mathcal{L}^q_s(\lambda)$ be the irreducible $\Uq(\mathfrak{q}_s)$-supermodule with highest weight $\lambda\in\Lambda_s^+\cap P_s^+$ and $m\leqslant s$. We define
$$\mathcal{L}^q_s(\lambda)^{\Uq(\mathfrak{q}_m)}:=\{v\in\mathcal{L}^q_s(\lambda)|\ k_i.v=v\text{ for }m<i\leqslant s\}.$$
Then $\mathcal{L}^q_s(\lambda)^{\Uq(\mathfrak{q}_m)}$ is a $\Uq(\mathfrak{q}_m)$-supermodule and 
$$\mathcal{L}^q_s(\lambda)^{\Uq(\mathfrak{q}_m)}\cong\begin{cases}\mathcal{L}^q_m(\lambda),&\text{if }\lambda_i=0\text{ for }m<i\leqslant s,\\
0,&\text{otherwise.}\end{cases}$$
\end{lemma}
\begin{proof}
We prove the case of $m=s-1$, then the lemma easily follows from an induction. Let $\mathbf{v}_{\lambda}\subseteq\mathcal{L}^q_s(\lambda)$ be the weight space of weight $\lambda$ and $M:=\Uq(\mathfrak{q}_{s-1}).\mathbf{v}_{\lambda}$. Then $M$ is a highest weight $\Uq(\mathfrak{q}_{s-1})$-supermodule of highest weight $\bar{\lambda}:=\lambda_1\epsilon_1+\cdots+\lambda_{s-1}\epsilon_{s-1}$. Note that $M$ is finite-dimensional, we deduce by Corollary~5.15 in \cite{GJKK} that $M$ is isomorphic to $\mathcal{L}^q_{s-1}(\bar{\lambda})$.

Since $k_s$ acts on $\mathbf{v}_{\lambda}$ as a scalar $q^{\lambda_s}$ and commutes with the action of $\Uq(\mathfrak{q}_{s-1})$ on $M$, we know that $k_s$ acts on $M$ as a scalar $q^{\lambda_s}$. Let $\mathfrak{j}^q_-$ be the sub-superalgebra of $\Uq(\mathfrak{q}_s)$ generated by $L_{i,s}$ for $i\in I_{s-1|s-1}$, then the PBW theorem implies that 
$$\mathcal{L}^q_s(\lambda)=\mathfrak{j}^q_-M.$$
It follows that the eigenvalues of $k_s$ on $\mathcal{L}^q_s(\lambda)$ are of the form $q^{\lambda_s+k}$ for $k\in\mathbb{Z}_+$ and $M$ is exactly the eigenspace of eigenvalue $q^{\lambda_s}$. 

Note that $\mathcal{L}^q_s(\lambda)^{\Uq(\mathfrak{q}_{s-1})}$ is the eigenspace of eigenvalue $1$ with respect to $k_s$, we conclude that $\mathcal{L}^q_s(\lambda)^{\Uq(\mathfrak{q}_{s-1})}=\mathcal{L}^q_{s-1}(\lambda)$ if $\lambda_s=0$ and $\mathcal{L}^q_s(\lambda)^{\Uq(\mathfrak{q}_{s-1})}=0$ if $\lambda_s>0$.
\end{proof}

Now, we return to the proof of Theorem~\ref{thm:howe}.

\begin{proof}[Proof of Theorem~\ref{thm:howe}]
Set $s=\max(m,n)$. We recall from Theorem~\ref{thm:Aq_dec} that the quantum coordinate superalgebra $\mathcal{A}_q(\mathfrak{q}_s)$ admits the multiplicity-free decomposition (\ref{eq:MfreeDec}) as a $\Uqi(\mathfrak{q}_s)\otimes\Uq(\mathfrak{q}_s)$-supermodule. Now, $\mathcal{A}_q(\mathfrak{q}_n,\mathfrak{q}_m)$ is a $\Uqi(\mathfrak{q}_n)\otimes\Uq(\mathfrak{q}_m)$-sub-supermodule of $\mathcal{A}_q(\mathfrak{q}_s)$ that consists of elements fixed by $\tilde{\Psi}_{k_i}$ for $n<i\leqslant s$ and $\Phi_{k_i}$ for $m<i\leqslant s$. It suffices to compute the sub-supermodule of each direct summand $\mathcal{L}^{q^{-1}}_s(\lambda)\circledast\mathcal{L}^q_s(\lambda)$ that consists of elements fixed  by $\tilde{\Psi}_{k_i}$ for $n<i\leqslant s$ and $\Phi_{k_i}$ for $m<i\leqslant s$. We denote it by 
\begin{align*}
&(\mathcal{L}^{q^{-1}}_s(\lambda)\circledast\mathcal{L}^q_s(\lambda))^{\Uqi(\mathfrak{q}_n)\otimes\Uq(\mathfrak{q}_m)}\\
:=&\{v\in\mathcal{L}^{q^{-1}}_s(\lambda)\circledast\mathcal{L}^q_s(\lambda)|\tilde{\Psi}_{k_i}v=v\text{ for }n<i\leqslant s, \text{ and }\Phi_{k_j}v=v\text{ for }m<j\leqslant s\},
\end{align*}
which is a $\Uqi(\mathfrak{q}_n)\otimes\Uq(\mathfrak{q}_m)$-sub-supermodule of $\mathcal{L}^{q^{-1}}_s(\lambda)^{\Uqi(\mathfrak{q}_n)}\otimes\mathcal{L}^q_s(\lambda)^{\Uq(\mathfrak{q}_m)}$.

By Lemma~\ref{lem:fixpt}, $\mathcal{L}^{q^{-1}}_s(\lambda)^{\Uqi(\mathfrak{q}_n)}\otimes\mathcal{L}^q_s(\lambda)^{\Uq(\mathfrak{q}_m)}$ vanishes unless $\lambda_i=0$ for all $i>n$ and $i>m$. Hence, we may assume $\lambda\in\Lambda_r^+\cap P_r^+$. In this case, 
$$\mathcal{L}^{q^{-1}}_s(\lambda)^{\Uqi(\mathfrak{q}_n)}\otimes\mathcal{L}^q_s(\lambda)^{\Uq(\mathfrak{q}_m)}\cong\mathcal{L}^{q^{-1}}_n(\lambda)\otimes\mathcal{L}^q_m(\lambda).$$

Now, if $\ell(\lambda)$ is even, then $\mathcal{L}^{q^{-1}}_s(\lambda)\circledast\mathcal{L}^q_s(\lambda)$ is isomorphic to $\mathcal{L}^{q^{-1}}_s(\lambda)\otimes\mathcal{L}^q_s(\lambda)$. Hence, 
$$(\mathcal{L}^{q^{-1}}_s(\lambda)\circledast\mathcal{L}^q_s(\lambda))^{\Uqi(\mathfrak{q}_n)\otimes\Uq(\mathfrak{q}_m)}\cong\mathcal{L}^{q^{-1}}_n(\lambda)\otimes\mathcal{L}^q_m(\lambda)=\mathcal{L}^{q^{-1}}_n(\lambda)\circledast\mathcal{L}^q_m(\lambda).$$

If $\ell(\lambda)$ is odd, then 
$$\mathcal{L}^{q^{-1}}_s(\lambda)\otimes\mathcal{L}^q_s(\lambda)\cong
(\mathcal{L}^{q^{-1}}_s(\lambda)\circledast\mathcal{L}^q_s(\lambda))^{\oplus2},$$
which yields that
$$\mathcal{L}^{q^{-1}}_n(\lambda)\otimes\mathcal{L}^q_m(\lambda)
\cong\mathcal{L}^{q^{-1}}_s(\lambda)^{\Uq(\mathfrak{q}_n)}\otimes \mathcal{L}^{q}_s(\lambda)^{\Uq(\mathfrak{q}_m)}
\cong\left((\mathcal{L}^{q^{-1}}_s(\lambda)\circledast \mathcal{L}^q_s(\lambda))^{\Uqi(\mathfrak{q}_n)\otimes\Uq(\mathfrak{q}_m)}\right)^{\oplus2}.$$
Hence, we conclude that
$$(\mathcal{L}^{q^{-1}}_s(\lambda)\circledast \mathcal{L}^q_s(\lambda))^{\Uqi(\mathfrak{q}_n)\otimes\Uq(\mathfrak{q}_m)}\cong
\mathcal{L}^{q^{-1}}_n(\lambda)\circledast\mathcal{L}^q_m(\lambda).$$
The desired decomposition for $\mathcal{A}_q(\mathfrak{q}_n,\mathfrak{q}_m)$ follows.
\end{proof}

\begin{remark}
As we have shown in Remark~\ref{rmk:AqA1}, the $\mathbf{A}_1$-sub-superalgebra of $\mathcal{A}_q(\mathfrak{q}_n)$ generated by $t_{ab}, a,b\in I_{n|n}$ is an $\mathbf{A}_1$-form of $\mathcal{A}_q(\mathfrak{q}_n)$. Similarly, the $\mathbf{A}_1$-sub-superalgebra $\mathcal{A}_{\mathbf{A}_1}(\mathfrak{q}_n,\mathfrak{q}_m)$ of $\mathcal{A}_q(\mathfrak{q}_n,\mathfrak{q}_m)$ generated by $t_{ab}$ for $a\in I_{n|n}$ and $b\in I_{m|m}$ is an $\mathbf{A}_1$-form of $\mathcal{A}_q(\mathfrak{q}_n,\mathfrak{q}_m)$, which is also invariant under the actions $\tilde{\Psi}$, $\Psi$ and $\Phi$ of $\UA(\mathfrak{q}_n)$. Moreover, $\mathbf{A}_1/\mathbf{J}_1\otimes_{\mathbf{A}_1}\mathcal{A}_{\mathbf{A}_1}(\mathfrak{q}_n,\mathfrak{q}_m)$ is isomorphic to the symmetric superalgebra $S(\mathbb{C}^{mn|mn})$. By taking the classical limits, Theorem~\ref{thm:howe} implies the $(\mathrm{U}(\mathfrak{q}_n),\mathrm{U}(\mathfrak{q}_m))$-Howe duality obtained in \cite{CW00}.
\end{remark}

\section{Sergeev duality for quantum queer superalgebras}

The Sergeev-Olshanski duality \cite{Olshanski} states that the $\Uq(\mathfrak{q}_n)$-supermodule $V_q^{\otimes m}$ admits an action of the finite Hecke-Clifford superalgebra $\mathrm{HC}_q(m)$, which centralizes the action of $\Uq(\mathfrak{q}_n)$. We will show in this section that the Sergeev-Olshanski duality is also implied by the $(\Uqi(\mathfrak{q}_n),\Uq(\mathfrak{q}_m))$-Howe duality.

The finite Hecke-Clifford superalgebra $\mathrm{HC}_q(m)$ is the unital associative superalgebra over $\mathbb{C}((q))$ with the even generators $T_1,\ldots, T_{m-1}$ and odd generators $C_1,\ldots,C_m$ subject to the following relations:
\begin{align}
&(T_a-q)(T_a+q^{-1})=0&&\text{for }a=1,\ldots,m-1,\tag{HC1}\label{eq:HC1}\\
&T_aT_{a+1}T_a=T_{a+1}T_aT_{a+1}&&\text{for }a=1,\ldots,m-2,\tag{HC2}\label{eq:HC2}\\
&T_aT_b=T_bT_a&&\text{for }a,b=1,\ldots,m-1\text{ and }|a-b|>1,\tag{HC3}\label{eq:HC3}\\
&C_a^2=1&&\text{for }a=1,\ldots,m,\tag{HC4}\label{eq:HC4}\\
&C_aC_b=-C_bC_a&&\text{for }a,b=1,\ldots,m\text{ and }a\neq b,\tag{HC5}\label{eq:HC5}\\
&T_aC_a=C_{a+1}T_a&&\text{for }a=1,\ldots,m-1,\tag{HC6}\label{eq:HC6}\\
&T_aC_b=C_bT_a&&\text{for }a=1,\ldots,m-1, b=1,\ldots,m,\text{ and }b\neq a,a+1.\tag{HC7}\label{eq:HC7}
\end{align}
The finite Hecke-Clifford superalgebra $\mathrm{HC}_q(m)$ is a quantum deformation of the Sergeev superalgebra in \cite{Sergeev}. The classification of finite-dimensional irreducible $\mathrm{HC}_q(m)$-supermodules was obtained in \cite{BK01}, in which $\mathrm{HC}_q(m)$ was viewed as a special cyclotomic Hecke-Clifford superalgebra. Every finite-dimensional irreducible $\mathrm{HC}_q(m)$ is determined by a strict partition $\lambda$ of $m$ up to the parity reversing functor $\Pi$. The irreducible $\mathrm{HC}_q(m)$-supermodule determined by $\lambda$ is denoted by $D^q_m(\lambda)$.

The tensor space $V_q^{\otimes m}$ is also an $\mathrm{HC}_q(m)$-supermodule under the action
\begin{align}
T_a.v_{i_1}\otimes\cdots\otimes v_{i_m}=&(-1)^{|i_a||i_{a+1}|}q^{\varphi(i_{a},i_{a+1})}v_{i_1}\otimes\cdots\otimes v_{i_{a-1}}\otimes v_{i_{a+1}}\otimes v_{i_a}\otimes v_{i_{a+2}}\otimes\cdots\otimes v_{i_m}\nonumber\\
&+\delta_{i_a<i_{a+1}}\xi v_{i_1}\otimes\cdots\otimes v_{i_m}\nonumber\\
&+(-1)^{|i_{a+1}|}\delta_{-i_a<i_{a+1}}\xi v_{i_1}\otimes v_{i_{a-1}}\otimes v_{-i_a}\otimes v_{-i_{a+1}}\otimes\cdots\otimes v_{i_m},\label{eq:HCtensor1}\\
C_b.v_{i_1}\otimes\cdots\otimes v_{i_m}=&(-1)^{|i_1|+\cdots+|i_{b-1}|+|i_b|}v_{i_1}\otimes\cdots\otimes v_{i_{b-1}}\otimes v_{-i_b}\otimes v_{i_{b+1}}\otimes\cdots\otimes v_{i_m},\label{eq:HCtensor2}
\end{align}
for $1\leqslant a\leqslant m-1, 1\leqslant b\leqslant m$ and $1\leqslant i_1,\ldots, i_m\leqslant n$. Then the Sergeev-Olshanski duality can be restated as follows.

\begin{theorem}[Sergeev-Olshanski Duality \cite{Olshanski,WW}]
\label{thm:Sergeev}
The actions of  $\Uq(\mathfrak{q}_n)$ and $\mathrm{HC}_q(m)$ on $V_q^{\otimes m}$ are mutual centralizers. Moreover, the $\Uq(\mathfrak{q}_n)\otimes\mathrm{HC}_q(m)$-module $V_q^{\otimes m}$ admits the multiplicity-free decomposition
\begin{equation*}
V_q^{\otimes m}\cong\bigoplus_{\lambda\in \mathrm{SP}(m)\atop\ell(\lambda)\leqslant n}\mathcal{L}^q_n(\lambda)\circledast D^q_m(\lambda),
\label{eq:dualsergeev}
\end{equation*}
where $\mathrm{SP}(m)$ is the set of strict partitions of $m$.
\end{theorem}

In order to prove that the $\big(\Uqi(\mathfrak{q}_n),\Uq(\mathfrak{q}_m)\big)$-Howe duality implies the quantum Sergeev-Olshanski duality, we need some preparation.

Let $M$ be a locally finite weight supermodule over $\Uq(\mathfrak{q}_m)$. We introduce the braid operators $T_a\in\mathrm{End}_{\mathbb{C}((q))}(M)$ for $a=1,\ldots,m-1$ as in \cite{Saito}:
$$T_av=\sum\limits_{i,j,k\in\mathbb{Z}_+}(-1)^{|j|}q^{k(k-j)-i(i-j+k)+j-1}e_a^{(i)}f_a^{(j)}e_a^{(k)}k_a^{k-i}k_{a+1}^{i-k}v\text{ for }v\in M,$$
 where $x^{(j)}=\frac{x^j}{[j]!}$ for $x\in \Uq(\mathfrak{q}_m)$ and $[j]=\frac{q^j-q^{-j}}{q-q^{-1}}$. The same arguments as in \cite{Saito} show that the operators $T_a$ for $a=1,\ldots, m-1$ satisfy the braid relations (\ref{eq:HC2}) and (\ref{eq:HC3}).

Now, we consider the zero weight space $M_0$ of $M$ (that is the weight space of $\epsilon_1+\cdots+\epsilon_m$)
$$M_0:=\{u\in M|k_i.u=qu\text{ for }i=1,\ldots,m\}.$$
It is easy to observe that $M_0$ is invariant under $T_a$ for $ a=1,\ldots, m-1,$ since 
$$k_a.T_av=q^{\mu_a-\mu_{a+1}}T_ak_av,\quad
k_{a+1}.T_av=q^{\mu_{a+1}-\mu_a}T_ak_{a+1}.v,\quad\text{and }k_bT_av=T_ak_b.v,$$
for $b\neq a,a+1$ and a weight vector $v$ of weight $\mu$. Moreover, the commutativity of $k_a,\,a=1,\ldots,m$ and $k_{\bar{b}},\,b=1,\ldots,m$ ensures that $M_0$ is also invariant under $C_b:=k_{\bar{b}}$ for $b=1,\ldots,m$. 

For $\lambda\in\Lambda^+_m\cap P^+_m$, the irreducible highest weight supermodule $\mathcal{L}^q_m(\lambda)$ has a nonvanishing zero weight space only if $\lambda$ is a partition of $m$.

\begin{lemma}
\label{lem:hecke}
The endomorphisms $T_a,\,a=1,\ldots, m-1$ and $C_b,\,b=1,\ldots,m$ satisfy the relations (\ref{eq:HC1})-(\ref{eq:HC7}) for $\mathrm{HC}_{q^{-1}}(m)$, and hence, define an $\mathrm{HC}_{q^{-1}}(m)$-supermodule structure on $M_0$.
\end{lemma}
\begin{proof}
We only check (\ref{eq:HC1}) and (\ref{eq:HC4})-(\ref{eq:HC7}). The relation (\ref{eq:HC1}) follows from the fact that 
$$k_{\bar{a}}^2=\frac{q^{2k_a}-q^{-2k_a}}{q^2-q^{-2}},$$
which acts on $M_0$ as identity. The relations (\ref{eq:HC5}) and (\ref{eq:HC7}) also hold obviously. 

In order to check (\ref{eq:HC1}) and (\ref{eq:HC6}), we set $\Uq(\mathfrak{q}_2)_a$ to be the sub-superalgebra of $\Uq(\mathfrak{q}_n)$ generated by $q^{\pm k_a}, q^{\pm k_{a+1}}, \bar{k}_{a}, \bar{k}_{a+1}, e_a, e_{a+1}, \bar{e}_{a}, \bar{e}_{a+1}, f_a, f_{a+1}, \bar{f}_{a}, \bar{f}_{a+1}$. Then $T_a, T_{a+1}, C_a, C_{a+1}$ lie in the image of $\Uq(\mathfrak{q}_2)_a$ in $\mathrm{End}(M)$. 

Now, $M$ is completely reducible as a $\Uq(\mathfrak{q}_2)_a$-supermodule. It suffices to verify (\ref{eq:HC1}) and (\ref{eq:HC6}) on the zero weight space of an irreducible highest weight $\Uq(\mathfrak{q}_2)_a$-supermodule $\mathcal{L}^q_2(\lambda)$. Note that the zero weight space of $\mathcal{L}^q_2(\lambda)$ is zero unless $\lambda=2\epsilon_1$. Hence, we only need to check (\ref{eq:HC1}) and (\ref{eq:HC6}) on the irreducible $\Uq(\mathfrak{q}_2)$-supermodule $\mathcal{L}^q_2(2\epsilon_1)$.

A straightforward computation shows that $\mathcal{L}^q_2(2\epsilon_1)$ has a basis $\{u_0,u_1,u_2, w, \bar{u}_0,\bar{u}_1,\bar{u}_2,\bar{w}\}$, on which $\Uq(\mathfrak{q}_2)_a$ acts as follows:
\begin{align*}
k_a.u_i&=q^{2-i}u_i,&
k_a.\bar{u}_0&=q^{2-i}\bar{u}_i,\\
k_a.w&=qw,&
k_a.\bar{w}&=q\bar{w},\\
k_{a+1}.u_i&=q^iu_i,&
k_{a+1}.\bar{u}_i&=q^i\bar{u}_i,\\
k_{a+1}.w&=qw,&
k_{a+1}.\bar{w}&=q\bar{w},\\
e_a.u_0&=0,&
e_a.\bar{u}_0&=0,\\
e_a.u_1&=(q+q^{-1})u_0,&
e_a.\bar{u}_1&=(q+q^{-1})\bar{u}_0,\\
e_a.u_2&=qu_1,&
e_a.\bar{u}_2&=q\bar{u}_1,\\
e_a.w&=0,&
e_a.\bar{w}&=0,\\
f_a.u_0&=u_1,&
f_a.\bar{u}_0&=\bar{u}_1,\\
f_a.u_1&=q^{-1}(q+q^{-1})u_2,&
f_a.\bar{u}_1&=q^{-1}(q+q^{-1})\bar{u}_2,\\
f_a.u_2&=0,&
f_a.\bar{u}_2&=0,\\
f_a.w&=0,&
f_a.\bar{w}&=0,\\
\bar{k}_{a}.u_0&=\bar{u}_0,&
\bar{k}_{a}.\bar{u}_0&=(q^2+q^{-2})u_0,\\
\bar{k}_{a}.u_1&=\frac{1}{q+q^{-1}}\bar{u}_1-q^2\bar{w},&
\bar{k}_{a}.\bar{u}_1&=\frac{q^2+q^{-2}}{q+q^{-1}}u_1-q^2w,\\
\bar{k}_{a}.u_2&=0,&
\bar{k}_{a}.\bar{u}_2&=0,\\
\bar{k}_{a}.w&=-\frac{q^2+q^{-2}}{q+q^{-1}}\bar{w}-\frac{2q^{-2}}{(q+q^{-1})^2}\bar{u}_1,&
\bar{k}_{a}.\bar{w}&=-\frac{1}{q+q^{-1}}w-\frac{2q^{-2}}{(q+q^{-1})^2}u_1,\\
\bar{k}_{a+1}.u_0&=0,&
\bar{k}_{a+1}.\bar{u}_0&=0,\\
\bar{k}_{a+1}.u_1&=\frac{1}{q+q^{-1}}\bar{u}_1+\bar{w},&
\bar{k}_{a+1}.\bar{u}_1&=\frac{q^2+q^{-2}}{q+q^{-1}}u_1+w,\\
\bar{k}_{a+1}.u_2&=\bar{u}_2,&
\bar{k}_{a+1}.\bar{u}_2&=(q^2+q^{-2})u_2,\\
\bar{k}_{a+1}.w&=-\frac{q^2+q^{-2}}{q+q^{-1}}\bar{w}+\frac{2}{(q+q^{-1})^2}\bar{u}_1,&
\bar{k}_{a+1}.\bar{w}&=-\frac{1}{q+q^{-1}}w+\frac{2}{(q+q^{-1})^2}u_1,\\
\bar{e}_a.u_0&=0,&
\bar{e}_a.\bar{u}_0&=0,\\
\bar{e}_a.u_1&=\bar{u}_0,&
\bar{e}_a.\bar{u}_1&=(q^2+q^{-2})u_0,\\
\bar{e}_a.u_2&=\frac{q}{q+q^{-1}}\bar{u}_1-q^3\bar{w},&
\bar{e}_a.\bar{u}_2&=q\frac{q^2+q^{-2}}{q+q^{-1}}u_1-q^3w,\\
\bar{e}_a.w&=\frac{2}{q+q^{-1}}\bar{u}_0,&
\bar{e}_a.\bar{w}&=\frac{2}{q+q^{-1}}u_0,\\
\bar{f}_a.u_0&=\frac{1}{q+q^{-1}}\bar{u}_1-q^2\bar{w},&
\bar{f}_a.\bar{u}_0&=\frac{q^2+q^{-2}}{q+q^{-1}}u_1+w,\\
\bar{f}_a.u_1&=q^{-1}\bar{u}_2,&
\bar{f}_a.\bar{u}_1&=q^{-1}(q^2+q^{-2})u_2,\\
\bar{f}_a.u_2&=0,&
\bar{f}_a.\bar{u}_2&=0,\\
\bar{f}_a.w&=-\frac{2q^{-3}}{q+q^{-1}}\bar{u}_2,&
\bar{f}_a.\bar{w}&=-\frac{2q^{-3}}{q+q^{-1}}u_2.
\end{align*}
Then the zero weight space $\mathcal{L}^q_2(2\epsilon_1)_0$ of $\mathcal{L}^q_2(2\epsilon_1)$ is spanned by $\{u_1, \bar{u}_1, w, \bar{w}\}$. Moreover, we have
$$T_au_1=-qu_1,\ \ T_a\bar{u}_1=-q\bar{u}_1,\ \ T_aw=q^{-1}w,\ \ T_a\bar{w}=q^{-1}w,$$
which implies that relations (\ref{eq:HC1}) and (\ref{eq:HC6}) for $\mathrm{HC}_{q^{-1}}(m)$.
\end{proof}

\begin{proposition}
\label{prop:Aq0wt}
Under the action $\Phi$ of $\Uq(\mathfrak{q}_m)$, the zero weight space $\mathcal{A}_q(\mathfrak{q}_n,\mathfrak{q}_m)_0$ of $\mathcal{A}_q(\mathfrak{q}_n,\mathfrak{q}_m)$ is isomorphic to $V_q^{\otimes m}$ as an $\mathrm{HC}_q(m)$-supermodule, where the $\mathrm{HC}_q(m)$-supermodule structure on $V_q^{\otimes m}$ is given by (\ref{eq:HCtensor1}) and (\ref{eq:HCtensor2}). 
\end{proposition}
\begin{proof}
The zero weight space of $\mathcal{A}_q(\mathfrak{q}_n,\mathfrak{q}_m)$ under the action $\Phi$ of $\Uq(\mathfrak{q}_m)$ is
$$\mathcal{A}_q(\mathfrak{q}_n,\mathfrak{q}_m)_0
=\{x\in\mathcal{A}_q(\mathfrak{q}_n, \mathfrak{q}_m)|\ \Phi_{k_i}(x)=qx\quad \text{for } 1\leqslant i\leqslant m\}.$$
Recall from Remark~\ref{rmk:Aqmn} that $\mathcal{A}_q(\mathfrak{q}_n,\mathfrak{q}_m)$ is spanned by 
$$\{t_{a_1,b_1}\cdots t_{a_l,b_l}|\ a_i\in I_{n|n},\,1\leqslant b_1\leqslant\cdots\leqslant b_l\leqslant m \text{ and }l\geqslant0 \}.$$
Under the action $\Phi$, 
$$\Phi_{k_i}(t_{a_1,b_1}\cdots t_{a_l,b_l})
=\Phi_{k_i}(t_{a_1,b_1})\cdots\Phi_{k_i}(t_{a_l,b_l})
=q^{\varphi(b_1,i)+\cdots+\varphi(b_l,i)}t_{a_1,b_1}\cdots t_{a_l,b_l}.$$
Hence, a monomial $t_{a_1,b_1}\cdots t_{a_l,b_l}$ is of eigenvalue $q$ with respect to $\Phi_{k_i}$ for all $i=1,\ldots,m$ if and only if $l=m$ and $b_i=i$ for $i=1,\ldots, m$. Therefore, $\mathcal{A}_q(\mathfrak{q}_n,\mathfrak{q}_m)_0$ is spanned by $\{t_{a_1,1}\cdots t_{a_m,m}|\ a_i\in I_{n|n}\}$.
Then a straightforward computation show that the $\mathbb{C}((q))$-linear map
$$\sigma: V_q^{\otimes m}\rightarrow\mathcal{A}_q(\mathfrak{q}_n,\mathfrak{q}_m)_0,\quad  v_{a_1}\otimes v_{a_2}\otimes\cdots\otimes v_{a_m}\rightarrow t_{a_1,1}t_{a_2,2}\cdots t_{a_m,m}$$
is an isomorphism of $\Uqi(\mathfrak{q}_n)\otimes \mathrm{HC}_{q^{-1}}(m)$-supermodules.
\end{proof}

\begin{proof}[Proof of Theorem~\ref{thm:Sergeev}]
By Theorem~\ref{thm:howe}, there is a multiplicity-free decomposition of $\Uqi(\mathfrak{q}_n)\otimes\Uq(\mathfrak{q}_m)$-supermodules
$$\mathcal{A}_q(\mathfrak{q}_n,\mathfrak{q}_m)\cong\bigoplus_{\lambda\in\Lambda^+_r\cap P^+_r}\mathcal{L}^{q^{-1}}_n(\lambda)\otimes\mathcal{L}^q_m(\lambda),$$
where $r=\min(m,n)$. Considering the zero weight spaces under the action $\Phi$ of $\Uq(\mathfrak{q}_m)$, we deduce from Lemma~\ref{lem:hecke} a decomposition of $\Uqi(\mathfrak{q}_n)\otimes\mathrm{HC}_{q^{-1}}(m)$-supermodules
$$\mathcal{A}_q(\mathfrak{q}_n,\mathfrak{q}_m)_0\cong\bigoplus_{\lambda\in\Lambda^+_r\cap P^+_r}\mathcal{L}^{q^{-1}}_n(\lambda)\otimes\mathcal{L}^q_m(\lambda)_0.$$
Now, Proposition~\ref{prop:Aq0wt} ensures that $\mathcal{A}_q(\mathfrak{q}_n,\mathfrak{q}_m)_0$ is isomorphic to $V_q^{\otimes m}$ as a $\Uqi(\mathfrak{q}_n)\otimes \mathrm{HC}_{q^{-1}}(m)$-supermodule. On the right hand side of the above decomposition, $\mathcal{L}^q_m(\lambda)_0$ vanishes unless $\lambda$ is a strict partition of $m$, in which case $\mathcal{L}^q_m(\lambda)_0$ is isomorphic to $D^{q^{-1}}_m(\lambda)$ as an $\mathrm{HC}_{q^{-1}}(m)$-supermodule. This can be proved by passing to the classical limits. The classical limit of the $\Uq(\mathfrak{q}_m)$-supermodule $\mathcal{L}^{q}_m(\lambda)$ is the $\mathrm{U}(\mathfrak{q}_m)$-supermodule $\mathbb{L}_m(\lambda)$, whose zero weight space has been shown to be the irreducible module over the Sergeev superalgebra $\mathrm{HC}_1(m)$ determined by $\lambda$. Now, we observe that $\mathcal{L}^{q}_m(\lambda)_0$ is an irreducible $\mathrm{HC}_{q^{-1}}(m)$-supermodule since the $\mathrm{HC}_1(m)$-supermodule $\mathbb{L}_m(\lambda)_0$ is irreducible. Hence, $\mathcal{L}^{q}_m(\lambda)_0$ is isomorphic to $D^{q^{-1}}_m(\lambda')$ for some strict partition $\lambda'$ of $m$. The partition $\lambda'$ of $m$ should equal $\lambda$ since $\mathcal{L}^{q}_m(\lambda)_0$ has the same character with $\mathbb{L}_m(\lambda)$.

Therefore, we obtain a multiplicity-free decomposition of $\Uqi(\mathfrak{q}_n)\otimes\mathrm{HC}_{q^{-1}}(m)$-supermodules
$$V_q^{\otimes m}=\bigoplus_{\lambda\in\mathrm{SP}(m)\atop\ell(\lambda)\leqslant n}\mathcal{L}^{q^{-1}}_n(\lambda)\otimes D^{q^{-1}}_m(\lambda).$$
Replace $q^{-1}$ with $q$, we obtain the desired decomposition.
\end{proof}

\section*{Acknowledgements}
We would like to express our debt to Chih-Whi Chen, Shun-Jen Cheng, Myungho Kim and Jinkui Wan for their suggestions. The paper was partially written up during the first author visit to NCTS in Taipei, from which we gratefully acknowledge the support and excellent working environment where most of this work was completed. Z. Chang also thanks the support of National Natural Science Foundation of China No. 11501213 and the China Postdoctoral Science Foundation No. 2016T90774.


\begin{thebibliography}{15}
\bibitem{BK01} J. Brundan and A. Kleshchev, Hecke-Clifford superalgebras, crystals of type $A_{2\ell}^{(2)}$ and modular branching rules for $\widehat{S}_n$, \textit{Represent. Theory}  {\bf 5}  (2001), 317-403.
\bibitem{BK} J. Brundan, A. Kleshchev, Projective representations of symmetric groups via Sergeev duality, \textit{Math. Z.} {\bf 239} (2002), 27-68.
\bibitem{CW01} S-J. Cheng, W. Wang, Howe Dualities for Lie superalgebras, \textit{Compositio Math.} {\bf 128} (2001), 55-94.
\bibitem{CW00} S-J. Cheng, W. Wang, Remarks on the Schur-Howe-Sergeev Duality, \textit{Lett. Math. Phy.} {\bf 52} (2000), 143-153.
\bibitem{CZ04} S-J. Cheng, R. Zhang, Howe duality and combinatorial character formula for orthosymplectic Lie superalgebras, \textit{ Adv. Math.}  {\bf 182}  (2004), 124-172. 
\bibitem{GJKK}  D. Grantcharov, J. H. Jung, S-J. Kang, M. Kim, Highest weight supermodules over quantum queer superalgebra $\text{U}_q(\mathfrak{q}(n))$, \textit{Comm. Math. Phys.}
 {\bf 296} (2010), 827-860. 
\bibitem{GJKKK} D. Grantcharov, J. H. Jung, S-J. Kang, M. Kashiwara, M. Kim, Crystal bases for the quantum queer superalgebra, \textit{J. Eur. Math. Soc.} {\bf 17}  (2015), 1593-1627. 
\bibitem{Howe1} R. Howe, Perspectives on invariant theory. The Schur Lectures (1992). Eds. I. Piatetski-Shapiro and S. Gelbart, Bar-Ilan University, 1995.
\bibitem{LZ} K. Lai and R. Zhang, Multiplicity-free actions of quantum groups and generalized Howe duality, \textit{Lett. Math. Phys.} {\bf 64} (2003), 255-272.
\bibitem{LZZ} G. I. Lehrer, H. Zhang and R. B. Zhang, A quantum analogue of the first fundamental theorem of classical invariant theory. \textit{Comm. Math. Phys.} 301 (2011), 131-174. 
\bibitem{Olshanski} G. Olshanski, Quantized universal enveloping superalgebra of type $Q$ and a super-extension of the Hecke algebra, \textit{Lett. Math. Phys.} {\bf 24} (2) (1992), 93-102.
\bibitem{Saito} Y. Saito, PBW basis of quantized universal enveloping algebra, \textit{Publ. RIMS. Kyoto Univ.} {\bf 30} (1990), 209-232.
\bibitem{WW} J. Wan, W. Wang, Frobenius character formula and spin generic degrees for Hecke-Clifford algebra, \textit{Proc. London Math. Soc.} {\bf 106} (2013), 287-317.
\bibitem{Sergeev} A. Sergeev, The Howe duality and the projective representations of symmetric groups, \textit{Represent. Theory} {\bf 3} (1999), 416-434.
\bibitem{WuZhang} Y. Wu and R. Zhang, Unitary highest weight representations of quantum general linear superalgebra, \textit{J. Algebra} {\bf 321} (2009), 3568-3593.
\bibitem{Zhang}R. Zhang, Howe duality and the quantum general linear group, \textit{Proc. Amer. Math. Soc.} {\bf 131} (2003), 2681-2692 .
\bibitem{Zhangyang}Y. Zhang, The first and second fundamental theorems of invariant theory for the quantum general linear supergroup, \textit{arXiv: 1703.01848}.
\end{thebibliography}
\end{document}